\documentclass[12pt, reqno]{amsart}
\usepackage{amsmath, amsthm, amscd, amsfonts, amssymb, graphicx, graphics,color, mathrsfs, epsfig, latexsym,cite}
\usepackage[plainpages,breaklinks,colorlinks=true,allcolors=blue,citecolor=red,pagebackref=false,hyperindex=false]{hyperref}
\definecolor{pink}{rgb}{1,0.03,0.8}
\usepackage[margin=2.9cm]{geometry}
\usepackage{t1enc}

\newtheorem{thm}{Theorem}[section]
\newtheorem{cor}[thm]{Corollary}
\newtheorem{lem}[thm]{Lemma}
\newtheorem{prop}[thm]{Proposition}
\theoremstyle{definition}

\theoremstyle{remark}
\newtheorem{rem}[thm]{Remark}
\numberwithin{equation}{section}
\newtheorem{example}[thm]{Example}

\newcommand{\V}{\mathscr{V}}

\newcommand{\X}{\mathcal{X}}

\newcommand{\K}{\mathscr{P}}

\begin{document}

\title[Asymptotic Behaviors and Hyperstability]{Hyers-Ulam Stability of Quadratic Operators in Locally Convex Cones}

\author [J.-H. Bae,  J. Mohammadpour and  A. Najati] { J.-H. Bae,  J. Mohammadpour and  A. Najati$^*$ }%

\address{J.-H. Bae
\newline
\indent School of Liberal Studies 
\newline \indent   Kyung Hee University
\newline \indent  Korea}
\email{jhbae@khu.ac.kr}
\address{Jafar Mohammadpour and Abbas Najati  
\newline
\indent Department of Mathematics
\newline
\indent Faculty of Mathematical Sciences
\newline \indent   University of Mohaghegh Ardabili
\newline \indent  Ardabil 56199-11367
\newline \indent Iran}
\email{jm8427836@gmail.com,~~ a.nejati@yahoo.com}

\thanks{$*$ Corresponding author}%

\begin{abstract}
The stability problem in Ulam's sense has recently been explored in locally convex cone environments, as shown in \cite{ MNF, NR1, NR2}. In continuation of this research direction, our work examines the stability properties of the quadratic functional  equation
\[
        2f\left(\frac{x+y}{2}\right) + 2f\left(\frac{x-y}{2}\right) = f(x) + f(y)
    \]
 in such structures. We present novel stability theorems that offer enhanced comprehension of operator behavior when subjected to perturbations. These results advance the theoretical framework of Hyers-Ulam stability within locally convex cones while elucidating distinctive characteristics of quadratic operators in this context. Our investigation both strengthens the mathematical underpinnings of stability theory and provides new perspectives on interactions between certain operators and locally convex spaces.
\end{abstract}

\maketitle

\section{Introduction}
The concept of locally convex cones was first introduced and further developed in \cite{KR} and \cite{Ro1}.
A cone is a mathematical structure defined by a set
$\K$ equipped with two operations: addition, denoted as $(a,b)\mapsto a+b$, and scalar multiplication, denoted as $(\lambda,a)\mapsto \lambda a$, where
$\lambda$ is a non-negative real number. The addition operation must satisfy the properties of associativity and commutativity, and there exists a neutral element
$0\in\K$ such that:
\begin{itemize}
  \item $(a+b)+c=a+(b+c)$ \quad\text{\rm for all}~ $a,b,c\in\K$,
  \item $a+b=b+a$ \quad \quad\text{\rm for all}~ $a,b\in\K$,
  \item $a+0=a$  \quad \quad\text{\rm for all}~ $a\in\K$.
\end{itemize}
Scalar multiplication must adhere to the standard associative and distributive rules, meaning:
\begin{itemize}
  \item $\lambda(\mu a)=(\lambda\mu)a\quad  \text{\rm for all}~\lambda,\mu\geqslant0~~\text{\rm and}~~ \text{\rm  all}~a\in\K$,
  \item $(\lambda+\mu)a=\lambda a+\mu a \quad  \text{\rm for all}~\lambda,\mu\geqslant0~~\text{\rm and}~~ \text{\rm  all}~a\in\K$,
  \item $\lambda(a+b)=\lambda a+\lambda b \quad  \text{\rm for all}~\lambda\geqslant0~~\text{\rm and}~~ \text{\rm  all}~a,b\in\K$,
  \item $1a=a \quad  \text{\rm for all}~a\in\K$,
  \item $0a=0 \quad  \text{\rm for all}~a\in\K$.
\end{itemize}
This structure ensures that the set $\K$ behaves consistently under both addition and scalar multiplication, maintaining the necessary algebraic properties.
Vector spaces defined over the field of real numbers naturally qualify as cones under this definition. However, cones differ from vector spaces in two key ways: first, they do not require the existence of additive inverses for their elements, and second, scalar multiplication is restricted to non-negative real numbers. Unlike in vector spaces, the property
$0a=0$ must be explicitly stated for cones, as it does not automatically follow from the other defining conditions. This distinction highlights the broader nature of cones compared to vector spaces, as they relax certain structural requirements while maintaining specific algebraic properties. This concept is illustrated by a straightforward example from [31]. Consider the set $\K=\{0,1\}$ with addition defined as $0+0=0, ~0+1=1+0=1$, and
$1+1=1$. Scalar multiplication by non-negative real numbers is defined as
$\lambda a=a$ for all $a\in\K$ and $\lambda\geqslant 0$. Clearly, all the previously mentioned axioms hold except for the last one, and in this case,
$0.1=1$.  It can be readily demonstrated that the neutral element $0\in\K$  for the addition operation is unique. The inclusion of the final axiom is crucial, as it ensures the property $\lambda0=0$  holds for all  $\lambda\geqslant0$ (see \cite[Proposition 1.1.1]{Ro3}).
\par
A subcone $\mathscr{Q}$ of a cone $\K$ is defined as a non-empty subset of $\K$ that remains closed under both the addition operation and scalar multiplication by non-negative real numbers. In other words, if $a,b\in \mathscr{Q}$ and $\lambda\geqslant0$, then $a+b\in \mathscr{Q}$ and $\lambda a\in \mathscr{Q}$.
\par
It is important to note that the (algebraic) cancellation law, which asserts that
\[a+c=b+c~~\text{\rm for}~a,b,c\in\K~~\text{implies}~a=b,\]
is not required in general. This property holds true if and only if the cone $\K$ can be embedded into a real vector space, as demonstrated in \cite[Theorem 1.1.4]{Ro3}.
\begin{example}
 Let $\Bbb{R}$ represent the field of real numbers. In the extended set  $\overline{\Bbb{R}}=\Bbb{R}\cup\{+\infty\}$
we adopt the standard algebraic operations, where $\lambda+(+\infty)=+\infty$ for all
$\lambda\in \overline{\Bbb{R}}, ~ \lambda.(+\infty)=+\infty$  for all  $\lambda>0$ and $0.(+\infty)=0$. Under these operations, $\overline{\Bbb{R}}$  forms a cone. However, since $\lambda+(+\infty)=\mu +(+\infty)$ for any  $\lambda,\mu\in \overline{\Bbb{R}}$,  the cancellation law does not hold in $\overline{\Bbb{R}}$.
Additionally, we use the notations $\overline{\Bbb{R}}_+=\{\lambda\in\overline{\Bbb{R}}:~\lambda\geqslant0\}$ and $\Bbb{R}_+=\{\lambda\in\Bbb{R}:~\lambda\geqslant0\}$ to denote the subcones of $\overline{\Bbb{R}}$ and $\Bbb{R}$, respectively.
\end{example}
\begin{example}
   Let $\K$ be a cone and $\mathcal{X}$  be an arbitrary non-empty set.  For $\K$-valued functions on
$\mathcal{X}$, addition and scalar multiplication can be defined pointwise. The collection
$\mathcal{F}(\mathcal{X},\K)$ of all such functions also forms a cone. In this case, the cancellation law applies to $\mathcal{F}(\mathcal{X},\K)$ if and only if it holds for the original cone $\K$.
\end{example}
A \textit{preordered cone} is a cone $\K $ equipped with a reflexive and transitive relation $\leqslant$ which is compatible with both the addition and scalar multiplication operations. Specifically, for all $a,b,c \in\K_1 $ and $\lambda\geqslant0$, the following properties hold:
\[
a\leqslant b \implies a+c \leqslant b+c \quad \text{and} \quad \lambda a \leqslant \lambda b.
\]
It is worth noting that the relation $\leqslant$ is not required to satisfy anti-symmetry, meaning that $a\leqslant b$ and $b\leqslant a$ do not necessarily imply
$a=b$.
As equality in $\K$ is obviously such an order, all results about ordered cones apply to cones without
order structures as well.
Every ordered vector space inherently satisfies the conditions to be considered a preordered cone. Additionally, the cones $\mathbb{\overline{R}} = \mathbb{R} \cup \{+\infty\}$ and $\mathbb{\overline{R}}_+ = [0,+\infty]$, when equipped with the standard order and algebraic operations (notably, $0 \cdot (+\infty) = 0)$ are examples of preordered cones.
Now, let  $(\K,\leqslant)$ be an preordered cone and $\X$ be an arbitrary non-empty set. The cone $\mathcal{F}(\X,\K)$
of all $\K$-valued functions on $\X$, when endowed with the pointwise order, also forms a preordered cone.
\par
Every cone $\K$ can be endowed with a natural preorder, defined by the relation
$a\leqslant b$ if there exists an element $c\in\K$ such that $a+c=b$. This preorder structure is inherent to the cone and does not require additional assumptions.
Every subcone of a preordered cone  $(\K,\leqslant)$ naturally inherits the structure of a preordered cone. Additionally, a form of weak order cancellation holds, as described below:
\begin{prop}\cite{KR}
 Let $(\K,\leqslant)$ be a preordered cone. If $a+c\leqslant b+c$ for some $a,b,c\in\K$, then $a+\varepsilon c\leqslant b+\varepsilon c$ for all $\varepsilon>0$.
\end{prop}
A subset $E$ of a preordered cone $(\K,\leqslant)$  is called decreasing (or increasing) if, for $b\in\K$ and some $a\in E$,  the condition  $b\leqslant a$ (or $a\leqslant b$)
implies that $b\in E$. Unions and intersections of arbitrary families of decreasing
(or increasing) sets are again decreasing (or increasing).  Furthermore, the complement $\K\setminus E$
of a decreasing subset $E$ of $\K$ is increasing. To see why, suppose $b\leqslant c$  for $c\in\K$ and
$b\in\K\setminus E$, if $c\notin\K\setminus E$, then $c\in E$, and since $E$ is decreasing, we get $b\in E$, which contradicts $b\in\K\setminus E$. A similar argument shows that the complement  $\K\setminus E$ of an increasing set $E\subseteq \K$ is decreasing.
\par
An element $\nu\in\K$ is called an upper bound (or lower bound) for a subset
$E$  of a preordered cone $(\K,\leqslant)$ if  $a\leqslant\nu$ (or $\nu\leqslant a$)
for all $a\in E$.  If such a bound $\nu$ belongs to $E$, it is referred to as the greatest element (or smallest element)  of $E$. However, since the preorder relation is not required to be antisymmetric, $E$ may contain multiple greatest or smallest elements, meaning uniqueness is not guaranteed.
\par
A subset $E$ of a preordered cone $(\K,\leqslant)$ is said to be order convex if, for any
$c\in\K$ and $a, b \in E$, the condition
$a\leqslant c\leqslant b$  implies that
$c\in E$. When an order convex set is multiplied by a positive scalar, the resulting set remains order convex. Additionally, the intersection of any collection of order convex sets is also order convex.
\par
A subset $\mathcal{V}$ of a preordered cone $\K$ is referred to as an \textit{abstract 0-neighborhood system} if it satisfies the following conditions:
\begin{itemize}
  \item [$(i)$] $0 < v$ \text{ for all } $v \in \mathcal{V}$;
  \item [$(ii)$] For any $u, v \in \mathcal{V}$,  there exists an element  $w \in \mathcal{V}$ such that
  $w \leqslant u$ \text{ and } $w\leqslant v$;
  \item [$(iii)$] $u + v \in \mathcal{V}$ \text{ and } $\lambda v \in \mathcal{V}$ \text{ whenever } $u, v \in \mathcal{V}$ \text{ and } $\lambda > 0$.
\end{itemize}
The elements
$v$ of $\V$ define upper and lower neighborhoods for the elements of $\K$ as follows:
\begin{align*}
  v(a) &= \{b \in\K_1:~ b \leqslant a + v\}\quad (\text{\rm upper neighborhood}), \\
 (a)v &= \{b \in\K_1:~   a \leqslant b + v\}\quad (\text{\rm lower neighborhood}).
\end{align*}
The intersection $v(a)v:= v(a)\cap(a)v$ defines the symmetric neighborhood of $a$. These neighborhoods generate the upper, lower, and symmetric topologies on $\K$, respectively.
All upper neighborhoods $v(a)$ are decreasing convex subsets of $\K$ and all
lower neighborhoods $(a)v$ are increasing convex. The symmetric neighborhoods
$v(a)v$ are both convex and order convex. This conclusion arises because
$v(a)v$ is the intersection of a decreasing set and an increasing set. We also observe that the following relationships hold for all
$a,b\in\K$  and $\lambda>0$:
\begin{enumerate}
  \item $\lambda v(a)=(\lambda v)(\lambda a)$,
  \item $\lambda (a)v=(\lambda a)(\lambda v)$,
  \item $\lambda v(a)v=(\lambda v)(\lambda a)(\lambda v)$,
  \item $v(a)+b\subseteq v(a+b)$,
  \item $(a)v+b\subseteq (a+b)v$,
   \item $v(a)v+b\subseteq v(a+b)v$.
\end{enumerate}
Each of the three topologies on
$\K$ can rightly be described as locally convex. In the upper topology, any open subset of
$\K$ is decreasing, as it is formed by unions of neighborhoods of its elements. Similarly, in the lower topology, all open subsets of
$\K$ are increasing. Consequently, the complements of these open sets exhibit the opposite behavior: subsets closed in the upper topology are increasing, while those closed in the lower topology are decreasing. This duality arises from the relationship between open and closed sets in these topologies.
\par
Let $(\K,\leqslant)$ be a preordered cone, and let $\V\subseteq\K$ be an abstract $0$-neighborhood system. The pair $(\K,\V)$ is referred to as a \textit{full locally convex cone}. Due to the inherent asymmetry of cones, asymmetric conditions naturally arise. For technical reasons, we require that the elements of a full locally convex cone  $(\K,\V)$ are bounded below.  This means that for every $a \in\K $ and $v \in \V$, there exists $\rho > 0$ such that $0 \leqslant a + \rho v$. Therefore,  a full locally convex cone $(\K,\V,\leqslant)$  is a preordered cone $(\K,\leqslant)$  that contains an $0$- neighborhood system $\V$ such that all its elements are bounded below.
Any subcone of a full locally convex cone $(\K,\V)$, even if it does not include the abstract $0$-neighborhood system $\V$, is termed a \textit{locally convex cone}.
 An element $a \in\K$ is called \textit{upper bounded} if, for every $v \in \V$,  there exists a positive scalar $\lambda > 0$ such that $a \leqslant \lambda v$. The element $a\in\K$ is said to be \textit{bounded} if it is both lower and upper bounded.  In a full locally convex cone the bounded elements therefore coincide
with the upper bounded ones.
 By defining $\xi = \{\varepsilon > 0 :~\varepsilon \in \Bbb{R}\}$, the pairs $(\Bbb{\overline{R}}, \xi)$ and $(\Bbb{\overline{R}}_+, \xi)$ form examples of full locally convex cones.
 \par
 Let $(\K , \V)$ be a locally convex cone, and let  $a \in\K_1$. The closure of $a$ is defined as the intersection of all its upper neighborhoods, that is,
\[
\overline{a}: = \cap\{v(a):~ v \in \V\}.
\]
It is straightforward to verify that $\bar{a}$ represents the closure of the singleton set $\{a\}$ under the lower topology \cite[Corollary I.3.5]{KR}.
A locally convex cone  $(\K , \V)$ is said to be \textit{separated}  if, for any $a,b\in\K$, the equality $\overline{a}=\overline{b}$
  implies $a=b$. In other words, distinct elements in $\K$ must have distinct closures. The locally convex cone $(\K , \V)$ is separated if
and only if the symmetric topology on $\K$ is Hausdorff (see \cite[Proposition I.3.9]{KR}).
\par
For a net $(a_i)_{i \in I}$ in $(\K,\V)$, we say that it converges to an element $a \in\K_1$  with respect to the symmetric relative topology of $\K$ if,
 for every $v\in\V$, there exists an index $i_0 \in I$ such that $a_i\leqslant a+ v$ for all $i \geqslant i_0$. The net $(a_i)_{i \in I}$ is called a (symmetric) \textit{Cauchy net} if for every $v \in \V$, there exists an index $i_0 \in I$ such that $a_i \leqslant a_j+ v$ for all $i,j \geqslant i_0$. Clearly, if a net converges, it must also be a Cauchy net. A locally convex cone $(\mathscr{P}, \mathcal{V})$ is said to be (symmetric)  \textit{complete} if every (symmetric) Cauchy net in $(\K,\V)$ converges to an element in $\K$.
A locally convex cone $(\K, \V)$ is called a \textit{uc-cone} if $\V = \{\lambda w : \lambda > 0\}$ for some $w \in \V$. In this case, the element $w$ is referred to as the \textit{generating element} of $\mathcal{V}$. If $(\K, \V)$ is a uc-cone and $\K$ is simultaneously a real vector space,
then $\K$ becomes a seminormed space under the symmetric topology of $(\K, \V)$. When $\V = \{\lambda w : \lambda > 0\}$, the seminorm $q:\K \to [0, +\infty]$ is defined as:
\[
q(a) = \inf\{\mu > 0 : \mu^{-1} a \in w(0)w\}, \quad a \in\K.
\]
If the symmetric topology on $\K$ is Hausdorff, then $q$ becomes a norm on $\K$ (see \cite{AR} for further details).
\par
The study of functional equations and their stability has been a central topic in mathematical analysis for several decades. Among these, the quadratic functional equation, which arises naturally in various areas of mathematics and physics, has garnered significant attention.
\par
Let $\K$ and $\mathscr{Q}$ be cones. A function $f:\K\to \mathscr{Q}$ is referred to as a quadratic mapping if it satisfies
\[f(x+y)+f(x-y)=2f(x)+2f(y)\]
for all $x,y\in\K$ such that $x-y\in\K$.  The study of functional equations has been a fundamental area of mathematical research, with applications in various fields such as physics, engineering, and economics. Among these, the quadratic functional equation holds significant importance due to its connection with quadratic forms, inner product spaces, and orthogonal additivity. The stability of functional equations, particularly in the sense of Hyers-Ulam, has been a major topic of investigation since the mid-20th century, leading to profound developments in both pure and applied mathematics. The quadratic functional equation  was first systematically studied by Jordan and von Neumann \cite{JV} in the context of characterizing inner product spaces. They proved that a normed space satisfies the parallelogram law (and hence is an inner product space) if and only if its norm satisfies the quadratic functional equation. This established a deep link between quadratic functions and geometric structures.
\par
The stability problem for functional equations originated from a question posed by  Ulam \cite{Ulam} in 1940, concerning the approximation of homomorphisms. Specifically, Ulam asked: If a function approximately satisfies a functional equation, can it be approximated by an exact solution? In 1941,  Hyers \cite{Hyers} provided a partial answer for additive Cauchy equations in Banach spaces, marking the birth of Hyers-Ulam stability theory.
\begin{thm}
Let $X $ be a normed space,  $Y $ a Banach space, and $ f : X \to Y $ a mapping satisfying
\[
\|f(x + y) - f(x) - f(y)\| \leqslant \varepsilon
\]
for some constant $ \varepsilon > 0 $ and for all $ x, y \in X $. Then  the limit
\[
A(x) = \lim_{n\to\infty}\frac{f(2^n x)}{ 2^{n}}
\]
exists for each $x \in X $, and  $ A: X \to Y $ is the unique additive function satisfying
\[
\|f(x) - A(x)\| \leqslant \varepsilon
\]
for all $ x \in X$.
\end{thm}
The stability theory of functional equations has witnessed significant developments through several key contributions.
In 1950 Aoki \cite{Aoki} extended Hyers' theorem to additive mappings, marking an important generalization. Two decades later, Th.M. Rassias \cite{Ras78} achieved a crucial advancement by removing the boundedness restriction on the Cauchy difference for linear mappings. This work was further refined by  G\u{a}vru\c{t}a \cite{Gav} who introduced general control functions to bound the Cauchy difference. These theoretical breakthroughs have successfully extended the stability concept to various other functional equations.
The stability of the quadratic functional equation was first investigated by Skof \cite{Skof83} and Cholewa \cite{Ch84}, who proved that if a function
$f$ from an abelian group $(G,+)$ to a Banach space $Y$ satisfies the inequality
\[\|f(x+y)+f(x-y)-2f(x)-2f(y)\|\leqslant\varepsilon\]
for some $\varepsilon\geqslant0$ and for all $x,y\in G$, then there exists a unique quadratic function
$Q:G\to Y$ such that $\|f(x)-Q(x)\|\leqslant\frac{\varepsilon}{2}$ for all $x\in G$.
 \par
In this work, we investigate the stability of quadratic mappings in the setting of locally convex cones. These structures generalize locally convex topological vector spaces by incorporating order-theoretic and convexity properties. Specifically, a locally convex cone is equipped with a preorder and a system of convex neighborhoods, which allows for a natural extension of classical functional-analytic results to ordered convex structures. The study of quadratic operators in such cones not only extends the classical theory but also provides new insights into the behavior of these operators in ordered and convex structures.
\par
The primary motivation for this work stems from the need to understand the stability properties of quadratic operators in non-linear and non-Archimedean settings. While the stability of quadratic mappings in Banach spaces and normed spaces has been extensively studied, the extension of these results to locally convex cones remains relatively unexplored. This paper aims to fill this gap by establishing Hyers-Ulam type stability results for quadratic operators in locally convex cones. Our approach leverages the structure of these cones, particularly their symmetric topology and the properties of bounded elements, to derive stability theorems.
\par
Our results generalize and extend previous work on the stability of quadratic mappings, providing a unified framework for studying these operators in locally convex cones. The techniques employed in this paper are rooted in the theory of locally convex cones and functional analysis, and they offer a new perspective on the stability of quadratic operators in ordered and convex structures.
\section{Main Results}
The following lemmas have  essential roles in our main results.
\begin{lem}\label{Lem.ab}
Let $ (\K, \V) $ be a separated locally convex cone, and let $a,b\in\K$. If $a\leqslant b+v$ and $b\leqslant a+v$ hold for every $v\in\V$, then $a=b$.
\end{lem}
\begin{proof}
  To prove $a=b$, it suffices to show that $\overline{a}=\overline{b}$. Let $x\in\overline{a}$. Then $x\leqslant a+\frac{1}{2}v$ for all $v\in\V$. Since $a\leqslant b+\frac{1}{2}v$, it follows that $x\leqslant b+v$ for all $v\in\V$,  which implies $x\in\overline{b}$. This demonstrates that $\overline{a}\subseteq\overline{b}$. By a similar argument, we can show that $\overline{b}\subseteq\overline{a}$, leading to $\overline{a}=\overline{b}$. Because $ (\K , \V) $ is a separated locally convex cone, we conclude that $a=b$.
\end{proof}
\begin{cor}
Let $ (\K , \V) $ be a separated locally convex cone, and let $a,b\in\K$. If $a+v= b+v$  for every $v\in\V$, then $a=b$.
\end{cor}
\begin{lem}\cite{NR1}\label{Lem.0}
Let $ (\K , \V) $ be a locally convex cone, and let $ a \in \K  $. Suppose $ \{\lambda_n\} $ is a sequence of non-negative scalars converging to $0$ as $ n \to \infty $. Then $a$ is bounded if and only if $ \lambda_n a $ converges to $0$ as $ n \to \infty $ in the symmetric topology.
\end{lem}
\begin{prop}
Consider a cone $\mathcal{Q}$ and a locally convex cone $ (\K , \V) $. Suppose a function $f:\mathcal{Q}\to\K$ satisfies
\begin{equation}\label{qq}
   2f\left(\frac{x+y}{2}\right) + 2f\left(\frac{x-y}{2}\right) =f(x) + f(y),\quad x,y,x-y\in \mathcal{Q}
\end{equation}
and $f(0)$ is bounded. Then $f$ is a quadratic mapping.
\end{prop}
\begin{proof}
 Substituting $x=y=0$ in \eqref{qq}, we obtain $2f(0)=f(0)$. Therefore, we deduce $f(0)=\frac{1}{2^n}f(0)$ for all $n\in\Bbb{N}$. Since $f(0)$ is bounded, by Lemma \ref{Lem.0}, we get $f(0)=0$. Next, setting $y=0$ in  \eqref{qq} and using $f(0)=0$, we derive  $4f\left(\frac{x}{2}\right)=f(x)$ for all $x\in \mathcal{Q}$. Thus,  \eqref{qq} yields
  \[f(x+y)+f(x-y)=2f(x)+2f(y),\quad x,y,x-y\in \mathcal{Q}.\]
\end{proof}
Now, we present the following key results:
\begin{thm}\label{t1.1}
    Let $(\K_1, \V_1)$ be a locally convex cone and $(\K_2, \V_2)$ a separated full locally convex cone which is complete under the symmetric topology. Assume that a mapping $ f:\K_1\to\K_2 $ satisfies
   \begin{equation}\label{t1.2}
        2f\left(\frac{x+y}{2}\right) + 2f\left(\frac{x-y}{2}\right) \in v(f(x) + f(y))v
    \end{equation}
    for some bounded element $ v \in \V_2 $ and for all $ x, y \in\K_1$ with $x-y\in\K_1$. If  $ f(0) $ is bounded, then there exists a unique quadratic function $ Q:\K_1\to\K_2 $ and a positive real number $ \gamma $ such that
    \[
        Q(x) \in (\gamma v)(f(x))(\gamma v),\quad  x \in\K_1.
    \]
\end{thm}
\begin{proof}
    Let $ v \in \V_2 $ be a bounded element satisfying the condition in \eqref{t1.2}. By setting $ y = 0 $ in \eqref{t1.2}, we obtain
    \[
        4f\left(\frac{x}{2}\right) \in v(f(x) + f(0))v
    \]
    for all $ x \in\K_1$. This implies the inequalities:
    \begin{align}\label{t1.3}
        4f\left(\frac{x}{2}\right) \leqslant f(x) + f(0) + v, \\
        f(x) + f(0) \leqslant 4f\left(\frac{x}{2}\right) + v.
    \end{align}
    Since $ f(0) $ is bounded, there exists a positive real number $ \lambda > 0 $ such that
    \begin{align}\label{t1.4}
        f(0) + \lambda v \geqslant 0 \quad \text{and} \quad f(0) \leqslant \lambda v.
    \end{align}
    Combining \eqref{t1.3} and \eqref{t1.4}, we derive
    \begin{align*}
        4f\left(\frac{x}{2}\right) \leqslant f(x) + (\lambda + 1)v, \\
        f(x) \leqslant 4f\left(\frac{x}{2}\right) + (\lambda + 1)v.
    \end{align*}
    Hence
      \begin{align}
        f(x) \leqslant \frac{f(2x)}{4} + \frac{\lambda + 1}{4}v, \label{t1.5} \\
        \frac{f(2x)}{4} \leqslant f(x)+ \frac{\lambda + 1}{4}v. \label{t1.5+}
    \end{align}
    Using mathematical induction on $n$, we extend the result to the following inequalities:
    \begin{align}
        f(x) \leqslant \frac{1}{4^n} f(2^n x) + \frac{1}{3}\left(1 - \frac{1}{4^n}\right)(\lambda + 1)v, \label{t1.6} \\
        \frac{1}{4^n} f(2^n x) \leqslant f(x) + \frac{1}{3}\left(1 - \frac{1}{4^n}\right)(\lambda + 1)v \label{t1.6+}
    \end{align}
    for all $x\in\K_1$ and all $n\in\Bbb{N}$. To establish \eqref{t1.6}, we first note that the inequality holds for
 $n=1$ as a direct consequence of \eqref{t1.5}.  Next, assume that \eqref{t1.6} is valid for some $n$ and  all $x\in\K_1$. Then
     \begin{align*}
        f(x)& \leqslant \frac{f(2x)}{4} + \frac{\lambda + 1}{4}v\quad \text{(by \eqref{t1.5})} \\
        &\leqslant \frac{1}{4^{n+1}} f(2^{n+1} x) + \frac{1}{12}\left(1 - \frac{1}{4^n}\right)(\lambda + 1)v + \frac{\lambda + 1}{4}v \quad \text{(by \eqref{t1.6})}  \\
        &= \frac{1}{4^{n+1}} f(2^{n+1} x) + \frac{1}{3}\left(1 - \frac{1}{4^{n+1}}\right)(\lambda + 1)v
    \end{align*}
     for all $x\in\K_1$. This completes the proof of \eqref{t1.6} for all $n\in\Bbb{N}$.
     To prove \eqref{t1.6+}, we begin by observing that the inequality holds for
 $n=1$ as a direct result of \eqref{t1.5+}. Now, suppose that \eqref{t1.6+} is true for  some $n$ and  all $x\in\K_1$.  Starting from \eqref{t1.6+}, we derive:
     \begin{align*}
  \frac{1}{4^{n+1}} f(2^{n+1} x) & \leqslant \frac{f(2x)}{4} + \frac{1}{12}\left(1 - \frac{1}{4^n}\right)(\lambda + 1)v\quad \text{(by \eqref{t1.6+})} \\
        &\leqslant  f(x)+ \frac{\lambda + 1}{4}v+\frac{1}{12}\left(1 - \frac{1}{4^n}\right)(\lambda + 1)v \quad \text{(by \eqref{t1.5+})}  \\
        &= f(x) + \frac{1}{3}\left(1 - \frac{1}{4^{n+1}}\right)(\lambda + 1)v
    \end{align*}
     for all $x\in\K_1$. This establishes \eqref{t1.6+} for all  $n\in\Bbb{N}$.
     Replacing $ x $ with $ 2^m x $ in \eqref{t1.6}  and \eqref{t1.6+} and multiplying both sides by $ \frac{1}{4^m} $,  we conclude
   	\begin{align}
				 \frac{1}{4^{m}} f(2^{m}x)\leqslant \frac{1}{4^{n+m}} f(2^{n+m}x)+\frac{\lambda +1}{3 \times 4^{m}}v, \label{t1.7} \\
\frac{1}{4^{n+m}} f(2^{n+m}x)\leqslant\frac{1}{4^{m}} f(2^{m}x)+\frac{\lambda +1}{3 \times 4^{m}}v \label{t1.7+}
			\end{align}
 for all $x\in\K_1$ and all  $m, n\in\Bbb{N}$.
  Given that  $ v $ is bounded,  Lemma \ref{Lem.0} ensures that the sequence $\left\{\frac{\lambda +1}{3 \times 4^{m}}v\right\}_m$ converges to zero as $m\to+\infty$.
  Let $ u \in \V_2 $ an arbitrary element. Then, there exists a natural number $N\in\Bbb{N}$ such that
  \[\frac{\lambda +1}{3 \times 4^{m}}v\leqslant u,\quad m>N.\]
  Using this result together with \eqref{t1.7} and \eqref{t1.7+}, we deduce:
  \begin{align*}
				 \frac{1}{4^{m}} f(2^{m}x)\leqslant \frac{1}{4^{n+m}} f(2^{n+m}x)+u, \quad
\frac{1}{4^{n+m}} f(2^{n+m}x)\leqslant\frac{1}{4^{m}} f(2^{m}x)+u
			\end{align*}
 for all $x\in\K_1$ and all  $m, n\in\Bbb{N}$ with $m>N$.
 This implies that the sequence $ \left\{ \frac{1}{4^n} f(2^n x) \right\}_n $ forms a Cauchy sequence in with respect to the symmetric topology. Because $ (\K_2, \V_2) $ is a separated space, the symmetric topology is Hausdorff, ensuring that the limit of this sequence is uniquely determined. We now define the mapping  $Q:\K_1\to\K_2$ as follows:
    \[
        Q(x) = \lim_{n \to \infty} \frac{1}{4^n} f(2^n x),\quad x\in\K_1.
    \]
    Thus, for $x\in\K_1$,  there is $n\in\Bbb{N}$ such that
    \[\frac{1}{4^n} f(2^n x)\leqslant Q(x)+\frac{v}{3},\quad Q(x)\leqslant \frac{1}{4^n} f(2^n x)+\frac{v}{3}.\]
    Using \eqref{t1.6} and \eqref{t1.6+}, we deduce
   \[ f(x) \leqslant Q(x)+ \frac{\lambda + 2}{3}v,\quad Q(x)\leqslant f(x)+ \frac{\lambda + 2}{3}v.\]
   Hence $Q(x) \in (\gamma v)(f(x))(\gamma v) $, where $ \gamma = \frac{\lambda + 2}{3} $.
We now prove $Q$ is a quadratic mapping on $\K_1$. Let $x,y\in\K_1$ such that $x-y\in\K_1$.
 Replacing $ x $ with $ 2^n x $ and $ y $ with $ 2^n y $ in \eqref{t1.2} and multiplying both sides by $ \frac{1}{4^n} $, we obtain
   \begin{align}
			2\left[\frac{1}{4^{n}}f\left(2^{n}\left( \frac{x+y}{2}\right)\right) + \frac{1}{4^{n}}f\left(2^{n}\left(\frac{x-y}{2}\right)\right) \right] \leqslant\frac{1}{4^{n}} f(2^{n}x)+\frac{1}{4^{n}} f(2^{n}y)+ \frac{1}{4^{n}}v, \label{t1.8} \\
\frac{1}{4^{n}} f(2^{n}x)+\frac{1}{4^{n}} f(2^{n}y)\leqslant	2\left[\frac{1}{4^{n}}f\left(2^{n}\left(\frac{x+y}{2}\right)\right) + \frac{1}{4^{n}}f\left(2^{n}\left(\frac{x-y}{2}\right)\right) \right] +\frac{1}{4^{n}}v \label{t1.8+}
		\end{align}
for all $n\in\Bbb{N}$. Let $u\in\V_2$. We can find $n\in\Bbb{N}$ such that
\[\frac{1}{4^{n}}v\leqslant u,\quad \frac{1}{4^{n}} f(2^{n}z)\leqslant Q(x)+u,\quad Q(x)\leqslant \frac{1}{4^{n}} f(2^{n}z)+u,\quad z\in\left\{x,y,\frac{x+y}{2}, \frac{x-y}{2}\right\}.\]
From \eqref {t1.8} and \eqref{t1.8+}, we deduce
  \begin{align*}
			2Q\left(\frac{x+y}{2}\right) + 2Q\left(\frac{x-y}{2}\right) \leqslant Q(x)+Q(y)+7u,  \\
Q(x)+Q(y)\leqslant	2Q\left(\frac{x+y}{2}\right) + 2Q\left(\frac{x-y}{2}\right) +7u.
		\end{align*}
By Lemma \ref{Lem.ab}, we infer that
\begin{equation}\label{t1.9}
2Q\left(\frac{x+y}{2}\right) + 2Q\left(\frac{x-y}{2}\right) = Q(x)+Q(y),\quad x,y,x-y\in\K_1.
\end{equation}
Since $ f(0) $ is bounded, by Lemma 1, we have $Q(0)=0$. Letting $y=0$ in \eqref{t1.9}, we get $4Q\left(\frac{x}{2}\right) =Q(x)$ for all $x\in\K_1$. Thus, \eqref{t1.9} implies
    \[
       Q(x+y)+Q(x-y)=2Q(x)+2Q(y),\quad x,y,x-y\in\K_1.
    \]
    To show the uniqueness of $ Q $, let $ q:\K_1\to\K_2 $ be another quadratic mapping satisfying $q(x) \in (\gamma v)(f(x))(\gamma v) $ for all $x\in\K_1$. Then,
    \begin{align*}
        f(x) \leqslant q(x) + \gamma v \quad \text{\rm and} \quad q(x) \leqslant f(x) + \gamma v,\quad x\in\K_1.
    \end{align*}
    Hence
    \begin{align}\label{t1.10}
       \frac{1}{4^{n}} f(2^{n}x) \leqslant q(x) + \frac{\gamma}{4^n} v \quad \text{\rm and} \quad q(x) \leqslant  \frac{1}{4^{n}} f(2^{n}x) + \frac{\gamma}{4^n} v,\quad x\in\K_1.
    \end{align}
    Let $u\in\V_2$ and $x\in\K_1$. We can find $n\in\Bbb{N}$ such that
\[\frac{\gamma}{4^{n}}v\leqslant u,\quad \frac{1}{4^{n}} f(2^{n}x)\leqslant Q(x)+u,\quad Q(x)\leqslant \frac{1}{4^{n}} f(2^{n}x)+u.\]
From \eqref {t1.10},  we deduce
\[  Q(x) \leqslant q(x) + 2u\quad \text{\rm and} \quad q(x) \leqslant  Q(x)+2u.\]
By Lemma \ref{Lem.ab}, we infer that $Q(x)=q(x)$. This proves the uniqueness of $Q$.
\end{proof}
\begin{cor} Consider a locally convex cone $ (\K , \V) $ and a mapping   $ f : \K \to  (\Bbb{\overline{R}},\xi)$ $($or $f : \K \to (\Bbb{\overline{R}}_+, \xi))$   that satisfies the following condition
\[
        2f\left(\frac{x+y}{2}\right) + 2f\left(\frac{x-y}{2}\right)\in \varepsilon \left(f(x) + f(y) \right) \varepsilon
    \]
for some $\varepsilon>0$ and for all $x, y \in \K$ with $x-y\in\K$, where  $ f(0)\ne +\infty$. Then there exists a unique quadratic mapping $Q : \K \to  \Bbb{\overline{R}}$ $($or $Q: \K \to \Bbb{\overline{R}}_+)$  such that
\begin{align*}
 Q(x)\in \frac{\lambda+2}{3}\varepsilon (f(x))\frac{\lambda+2}{3}\varepsilon
\end{align*}
holds for all $ x \in \K $, with $\lambda>0$ chosen such that $|f(0)|\leqslant\lambda\varepsilon$.
\end{cor}
\begin{rem}
    Let $(\K_1, \V_1)$ be a locally convex cone and $(\K_2, \V_2)$ a separated full uc-cone with the generating element $v$.
    Assume $(\K_2, \V_2)$ is complete under the symmetric topology.  Suppose a mapping $ f:\K_1\to\K_2 $ satisfies
    \[
        2f\left(\frac{x+y}{2}\right) + 2f\left(\frac{x-y}{2}\right) \in v(f(x) + f(y))v
    \]
    for all $ x, y \in\K_1$ with $x-y\in\K_1$, where $ f(0) $ is bounded. Then, by Theorem \ref{t1.1},  there exists a unique  quadratic mapping $ Q:\K_1\to\K_2 $ and a positive real number $ \gamma $ such that
    \[
        Q(x) \in \gamma v(f(x))\gamma v
    \]
    for all $ x \in\K_1$.
\end{rem}
\begin{cor}
Consider a locally convex cone $(\K_1, \V_1)$ and a separated full uc-cone $(\K_2, \V_2)$, where $\K_2 $ is additionally equipped with the structure of a real vector space. Suppose $(\K_2, \V_2)$  is complete with respect to its symmetric topology. If for some $ \varepsilon > 0 $, a given mapping $ f:\K_1\to\K_2 $ satisfies the following condition:
\[
        2f\left(\frac{x+y}{2}\right) + 2f\left(\frac{x-y}{2}\right) \in \varepsilon v(f(x) + f(y))\varepsilon v
    \]
for all $ x, y \in\K_1$ with $x-y\in\K_1$, where $ v $ denotes the generating element of $ \V_2 $, then there exists a unique quadratic mapping $ Q:\K_1\to\K_2 $ satisfying
    \[Q(x)\in \left(\frac{r\varepsilon}{3} v\right)\left(f(x)-\frac{1}{3}f(0)\right)\left(\frac{r\varepsilon}{3} v\right),\quad x \in \K _1 \]
    for all $ r>1$.
\end{cor}
\begin{proof}
The symmetric topology on $\K_2$ is known to be Hausdorff (see \cite{KR}, I.3.9), which implies that  the pair $(\K_2, \V_2)$ becomes a Banach space when endowed with the norm
$\rho$ defined as:
\[
\rho(a) = \inf\{\mu > 0:~ \mu^{-1} a \in v(0)v\}, \quad a \in\K_2.
\]
Furthermore, the vector space structure of $\K_2$  guarantees that all its elements are bounded. Consequently,  the quantity 
$\rho(a)$ remains finite for every $a\in\K_2$.
  From our initial assumptions, we obtain
   \[2f\left(\frac{x+y}{2}\right) + 2f\left(\frac{x-y}{2}\right) - f(x) - f(y)\in (\varepsilon v)(0)(\varepsilon v)\]
   for all $ x, y \in\K_1$  with $x-y\in\K_1$.  This relationship can be alternatively formulated as
    \[
        \rho\left(2f\left(\frac{x+y}{2}\right) + 2f\left(\frac{x-y}{2}\right) - f(x) - f(y)\right) \leqslant \varepsilon
    \]
    for all $ x, y \in\K_1$  with $x-y\in\K_1$. By employing Hyers' approach and taking $ y = 0 $, we obtain 
    \[
        \rho\left(4f\left(\frac{x}{2}\right) - f(x) - f(0)\right) \leqslant \varepsilon,\quad x\in\K_1.
    \]
    Hence,
\begin{align}\label{p1.cor}
\rho\left(\frac{1}{4^{n+1}}f(2^{n+1}x)-\frac{1}{4^{m}}f(2^{m}x)+\sum_{k=m}^{n}\frac{1}{4^{k+1}}f(0)\right)\leqslant \sum_{k=m}^{n}\frac{\varepsilon}{4^{k+1}}
\end{align}
for all $ x \in \K _1 $ and $n\geqslant m\geqslant 0$. Based on  \eqref{p1.cor}, we observe that the sequence 
 $\left\{\frac{1}{4^{n}}f(2^{n}x)\right\}_{n\geqslant 1}$  is Cauchy in the Banach space
$(\K_2, \rho)$. Since the space $\K_2$ is complete with respect to the norm $\rho$, this sequence must converge. This allows us to introduce a well-defined mapping $Q : \K_1\to \K_2$ through the pointwise limit:
\[Q(x)=\lim_{n\to\infty}\frac{1}{4^{n}}f(2^{n}x),\quad x \in \K _1.\]
Setting $m=0$ and letting $n$ approach infinity in \eqref{p1.cor} yields
\[\rho\left(Q(x)-f(x)+\frac{1}{3}f(0)\right)\leqslant\frac{\varepsilon}{3},\quad x\in\K_1.\]
   Consequently, for all $x \in \K _1$ and $r>1$, we have   $Q(x)-f(x)+\frac{1}{3}f(0)\in \frac{r\varepsilon}{3} (v(0)v)$. This implies that
\[Q(x)\in \left(\frac{r\varepsilon}{3} v\right)\left(f(x)-\frac{1}{3}f(0)\right)\left(\frac{r\varepsilon}{3} v\right),\quad x \in \K _1 ,~ r>1.\] 
\end{proof}


\begin{thebibliography}{1}
\bibitem{Aoki}
T. Aoki, \textit{On the stability of the linear transformation in Banach spaces}, J. Math. Soc. Japan \textbf{2} (1950), 64--66.

\bibitem{AR}
D. Ayaseh and A. Ranjbari,  \textit{Bornological locally convex cones}, Matematiche (Catania) \textbf{69} (2014), no. 2, 267--284.

\bibitem{Ch84}
P. W. Cholewa,  \textit{Remarks on the stability of functional equations},  Aequationes Math. \textbf{27} (1984), no. 1-2, 76--86.

\bibitem{Gav}
P. G\u{a}vru\c{t}a,  \textit{A generalization of the Hyers-Ulam-Rassias stability of approximately additive mappings}. J. Math. Anal. Appl. \textbf{184} (1994), no. 3, 431--436.

\bibitem{JV}
P. Jordan and J.Von Neumann, \textit{On inner products in linear, metric spaces}, Ann. of Math. (2) \textbf{36} (1935), no. 3, 719--723.

\bibitem{Hyers}
D. H. Hyers, \textit{On the stability of the linear functional equation}, Proc. Natl. Acad. Sci. USA \textbf{27} (1941), 222–224.

\bibitem{KR}
K. Keimel and W. Roth, \textit{Ordered cones and approximation}, Lecture Notes in Mathematics 1517, Springer-Verlag, Berlin, 1992.


\bibitem{MNF}
J. Mohammadpour,  A. Najati and Iz. EL-Fassi, \textit{Hyers-Ulam type stability of the Pexiderized Cauchy functional equation in locally convex cones}, submitted.
   
   \bibitem{NR1}
I. Nikoufar and A. Ranjbari,  \textit{Stability of linear operators in locally convex cones}, Bull. Sci. Math. \textbf{191} (2024), Paper No. 103380, 9 pp.

\bibitem{NR2}
I. Nikoufar and A. Ranjbari, \textit{Hyers-Ulam type stability of Jensen operators in
locally convex cones}. J. Math. Inequal. \textbf{18}(4) (2024), 1303--1312. 

    \bibitem{Ras78}
   Themistocles M. Rassias,  \textit{On the stability of the linear mapping in Banach spaces}, Proc. Amer. Math. Soc. \textbf{72} (1978), no. 2, 297--300.

    \bibitem{Ro1}
W. Roth, \textit{Operator-valued measures and integrals for cone-valued functions}, Lecture Notes in Mathematics, vol. 1964. Springer-Verlag, Berlin (2009)

 \bibitem{Ro3}
W. Roth, \textit{Locally Convex Cones}, preprint.

\bibitem{Skof83}
F. Skof,  \textit{Local properties and approximation of operators}. (Italian) Geometry of Banach spaces and related topics (Italian) (Milan, 1983). Rend. Sem. Mat. Fis. Milano \textbf{53} (1983), 113--129 (1986).


\bibitem{Ulam}
S. M. Ulam, \textit{Problems in Modern Mathematics}, Science Editions, Wiley, New York, 1964.
\end{thebibliography}
\end{document}